\documentclass[a4paper, 12pt]{amsart}
\input xy
\xyoption{all}
\swapnumbers
\usepackage{amssymb}
\usepackage{amsthm}
\usepackage{enumerate}
\usepackage[active]{srcltx}
\usepackage{amsmath}
\usepackage{nicefrac}
\usepackage[T1]{fontenc}
\usepackage[ansinew]{inputenc}
\usepackage[colorlinks=true]{hyperref}

\numberwithin{equation}{section}
\theoremstyle{plain}
  \begingroup
        \newtheorem{theorem}[equation]{Theorem}
        \newtheorem{lemma}[equation]{Lemma}
        \newtheorem{proposition}[equation]{Proposition}
        \newtheorem{corollary}[equation]{Corollary}

	    \newtheorem{defi}[equation]{Definition}
        
        \newtheorem{convention}[equation]{Convention}
        
  \endgroup

\newcommand{\comments}[1]{}

\theoremstyle{definition}
  \begingroup
        \newtheorem{remark}[equation]{Remark}
        \newtheorem{example}[equation]{Example}
        \newtheorem{examples}[equation]{Examples}

       \newtheorem{notation}[equation]{Notation}
\newtheorem*{theorem*}{Theorem}

  \endgroup

\newcommand{\R}{\mathbb{R}}

\def\xto#1{\xrightarrow{#1}}
\def\HP{\underset{G}{\stackrel{\hspace{3pt} _2}{\mbox{\fontsize{24.88}{30} $\ast$}}}}

\def\2P#1{\underset{#1}{\stackrel{\hspace{4pt} _2}{\mbox{\fontsize{24.88}{30} $\ast$}}}}
\def\SP#1{\underset{#1}{\stackrel{\hspace{4pt} }{\mbox{\fontsize{24.88}{30} $\oplus$}}}}

\def\PR#1{\underset{#1}{\mbox{\fontsize{24.88}{30} $\ast$}}}

\begin{document}
\title[2-nilpotent product of groups]{Permanence properties of the second nilpotent product of groups}

\author{Rom\'an Sasyk}

\address{
Instituto Argentino de Matem\'aticas-CONICET\\
Saavedra 15, Piso 3 (1083), Buenos Aires, Argentina}
\address{
and}
\address{Departamento de Matem\'atica, Facultad de Ciencias Exactas y Naturales, Universidad de Buenos Aires, Argentina}

\email{rsasyk@dm.uba.ar}


\date{}


\maketitle

\begin{abstract}
We show that amenability,  the Haagerup property,  the Kazhdan's property (T) and exactness
are preserved under taking second nilpotent product of groups.
We also define the restricted second nilpotent wreath product of groups, this is a semi-direct product akin to the restricted wreath product but constructed from the second nilpotent product.
We then show that if two discrete groups have the Haagerup property, 
the restricted second nilpotent wreath product of them also has the Haagerup  property. We finally show that
if a discrete group is abelian, then the restricted second nilpotent wreath product constructed from it is unitarizable if and only if the acting group is amenable.

\end{abstract}
\section{Introduction}

Given a family of groups, the direct sum and the free product provide ways of constructing new groups out of them.
Even though both operations are quite different, they share the next common properties
\begin{enumerate}
\item associativity;\label{1 intro}
\item commutativity;\label{2 intro}
\item the product contains subgroups which generate the product;\label{3 intro}
\item these subgroups are isomorphic to the original groups;\label{4 intro}
\item the intersection of a given one of these subgroups with the normal subgroup generated by the rest of these subgroups is the identity.\label{5 intro}
\end{enumerate}

In the first edition of his classic book ``Theory of Groups'', Kurosh asked if there were other operations on a family of groups that satisfy the above properties. This problem was solved in the affirmative by Golovin in \cite{Gol1}, where he defined countably many such operations. Among these operations, the {\it second nilpotent product} stands out as the simplest to scrutinize. It is defined as follows
\begin{defi}
For a family of groups $\{H_i\}_{i \in \mathcal I}$ indexed on a set $\mathcal I$, the second nilpotent product of the family is the group
$$\2P {i \in \mathcal I} H_i :=  \nicefrac{\PR {i \in \mathcal I} H_i}{\big\langle [\PR {i \in \mathcal I} H_i,[H_j,H_k]]_{j \neq k}\big\rangle}$$
\end{defi}

\vskip .1in
The importance of having such operation resides in that it provides a ``new'' way of constructing groups. Unlike the direct sum, the second nilpotent product of a family of abelian groups does not need to be abelian. On the other hand, the second nilpotent product of finitely many finite groups is necessarily a finite group.

Second nilpotent products of groups have had an interesting application in Mathematical Logic. Indeed, they were used, albeit indirectly, by Mekler in \cite{Mekler} to show that the isomorphism relation of countable groups is model complete for countable structures. This was later framed in the context of Borel Reducibility in the pioneering article of Friedman and Stanley \cite{FS}. Inspired by these results, a variant of a construction of Mekler involving certain semi-direct product of groups related to the restricted wreath product of groups was developed by  Törnquist and the author in \cite{SasTor1} to show 
non-classification results for von Neumann algebras.  
The original motivation of the present article was to further analyze that construction. 
Here we study permanence properties of the second nilpotent product of groups that come from representation theory and dynamics of group actions.
Our first result is
\begin{theorem}[A]\label{theo1}
Amenability, the Haagerup's approximation property, the Kazhdan's property (T) and exactness are preserved under taking second nilpotent product of two countable groups.
\end{theorem}
This, together with the associativity of the second nilpotent product, will allow us to prove the next corollary. 

\begin{corollary}[B]
\label{amenability of many products}
If $\{H_i\}_{i\in\mathcal I}$ is a countable family of discrete amenable (respectively Haagerup, resp. exact) groups, the group $\2P {i\in\mathcal I} H_i$ is
amenable (resp. Haagerup, resp. exact).
\end{corollary}

In geometric and measurable group theory, restricted wreath products have been playing a significant role in producing examples of groups that verify or disprove long standing conjectures. Since the second nilpotent product of groups allows a construction similar to the wreath product, that we shall call
{\it restricted second nilpotent wreath product}, (see section \ref{sec5}), it is conceivable that  this new construction could be used in a similar manner. To illustrate this, we present the following adaptation of a theorem of Cornulier, Stalder \& Valette that appeared in \cite{CSV1,CSV2}.

\begin{theorem}[C]\label{theo2}
Let $G$ and $H$ be countable groups with the Haagerup property. Then the restricted second nilpotent wreath product 
$$\Big (\2P { G}H\Big )\rtimes G$$
 has the Haagerup property.
\end{theorem}
  In order to  explain our  last result, first 
 recall that a group $G$ is unitarizable if for every uniformly bounded representation $\pi$ of $G$ on a Hilbert space $\mathcal H$, there exists $T \in \mathcal B(\mathcal H)$ such that $T\pi T^{-1}$ is a unitary representation. Dixmier showed that amenable groups are always unitarizable, \cite{Dix}. The Dixmier problem asks whether the converse holds true (see, for instance, \cite{MR2195457}).
 In \cite{MoOz}, Monod \& Ozawa showed that wreath products of the form  $\Big(\SP {G}A\Big )\rtimes G$ with $A$ abelian, are unitarizable if and only if $G$ is amenable. This provided the first examples of non unitarizable groups not containing a copy of $\mathbb F_2$.   
Observe that a consequence of Corollary \ref {amenability of many products} 
is that if  a group $A$ is abelian,
 the group $\Big (\2P {G}A \Big )\rtimes G$
  is amenable if and only if $G$ is amenable.
The construction of the restricted second nilpotent product combined with  \cite[Theorem 1]{MoOz} will yield the following result.
\begin{theorem}[D]\label{theo3}
Let $A$ and $G$ be countable groups, with $A$ abelian. The restricted second nilpotent wreath product $\Big (\2P {G}A\Big )\rtimes G$ is unitarizable if and only if $G$ is amenable.
\end{theorem}
 We will see that for an abelian group $A\neq \{1\}$,  the group  $\2P { G}A $ is never abelian (unless  $G=\{1\}$), and that  $\2P { G}\nicefrac{\mathbb Z}{p\mathbb Z}$ is a nil-2 p-group. Hence, our examples are apparently different from the ones that have already appeared in the
literature.
\vskip .2in

To the best of our knowledge, the papers of Golovin have not been studied much in recent years. A partial indication of this is the scarce number of citations they received.  Thus, another goal of this article is to bring back the second nilpotent product as a useful construction of groups.

\vskip .1in

\section{The second nilpotent product of two groups}

We record several properties of the second nilpotent product of two groups. Some of the results of this section can be found (at least implicitly) in the papers of Golovin, \cite{Gol1,Gol2}. Since that articles are not widely available and the proofs given there are somewhat obscured by lack of modern terminology, we opted to provide short proofs of them.
\begin{convention}
In this article we adopt the convention $$[a,b]=aba^{-1}b^{-1}.$$
\end{convention}

\begin{defi}
Given two groups $A$ and $B$, their second nilpotent product is the group

$$A \stackrel{_2}{\ast} B:= \nicefrac{A \ast B}{[A \ast B,[A,B]]}$$
\end{defi}

In words, this means that we start from the free product and declare that all the commutators
$[a,b]$ with $a \in A$ and $b \in B$ are central. Historically, the construction of the second nilpotent product of exactly two groups first appeared in a paper of Levi, \cite{Levi}. Golovin, unaware of Levi's work, treated the general case in \cite{Gol1,Gol2} to solve the problem of Kurosh mentioned in the introduction.

\begin{notation} We denote with $[A,B]^{(2)}$ the subgroup of $A \stackrel{_2}{\ast} B$ generated by the commutators of the form $[a,b]$, with $a \in A$ and
$b \in B$.
\end{notation}

\begin{proposition} \label{abc}
 An element $x \in A \stackrel{_2}{\ast} B$ admits a unique representation $abc$ where $a \in A$, \mbox{$b \in B$}
 and  $c \in [A,B]^{(2)}$.
\end{proposition}
\begin{proof}
 The element $x$ can be represented as a word with letters in $A$ and $B$
 $$x=a_1 b_1 a_2 b_2 ... a_n b_n$$
 The identity $ba=[b,a]ab$ allows us to take all the $a_i$ to the left. Having in mind that
 $[b_j, a_i]$ are central  in $A \stackrel{_2}{\ast} B$, we obtain $x=a_1...a_nb_1...b_nc$ with \mbox{$c \in [A,B]^{(2)}$}.
 Uniqueness follows easily by applying the projections \mbox{$A \stackrel{_2}{\ast} B \xto{\pi_A} A$} and
 $A \stackrel{_2}{\ast} B \xto{\pi_B} B$.
\end{proof}
\begin{corollary} We have the following exact sequence
\begin{equation}\label{SESeqa}
1 \to [A,B]^{(2)} \to A \stackrel{_2}{\ast} B \to A \oplus B \to 1.\\
\end{equation}

\end{corollary}
\begin{proposition} \label{morfismo}
 $B \xto{[a,-]} [A,B]^{(2)}$ is a group homomorphism for every element $a \in A$.
\end{proposition}

 \begin{proof}
  Let $b, c \in B$. The identity $[a,bc]=[a,b]b[a,c]b^{-1}$ holds true in the free product $A{\ast} B$.
It follows that  $[a,bc]=[a,b][a,c]$ in  $A \stackrel{_2}{\ast} B$.
\end{proof}
\begin{corollary}\label{commuting}  Let $a\in A$ and $b, c \in B$. Then $[a,bc]=[a,cb]$.
\end{corollary}
\begin{proof} By Proposition \ref{morfismo} we have that 
$[a,bc]=[a,b][a,c]=[a,c][a,b]=[a,cb].$
\end{proof}

\begin{corollary}\label{commuteswithcommutator}
$A$ commutes with $[B,B]$  inside $A \stackrel{_2}{\ast} B$.
\end{corollary}
 \begin{proof}
 By Proposition \ref{morfismo} and Corollary \ref{commuting} we have that 
 $$[a,[b,b']]=[a,bb^\prime][a,b^{-1}b^{\prime -1}]=[a,b^\prime b][a,b^{-1}b'^{-1}]=[a,b^\prime bb^{-1}b^{\prime -1}]=1.$$
\end{proof}

In order to better understand the second nilpotent product, it is convenient to use the tensor product of groups (not necessarily abelian). This tensor product and the corresponding results were already introduced by Whitney in \cite{Whit}.

\begin{defi}
  Let $A$ and $B$ be two groups. The tensor product $A \tilde{\otimes} B$ is defined as

$$ \mbox{ \LARGE $\nicefrac{ {  \mathbb F_{\tiny A \times B} } }{ {\tiny {\begin{matrix} (a_1a_2, b) \sim (a_1, b) (a_2, b) \\ (a, b_1b_2) \sim (a, b_1) (a, b_2) \end{matrix}}}} $ } $$

\noindent where $\mathbb{F}_{A \times B}$ denotes the free group generated by the set $A\times B$.
\end{defi}
The natural application $A \times B \to A \tilde{\otimes} B$ is a group homomorphism in each variable, 
and for every group $G$ and every homomorphism in each variable $\varphi:A\times B\to G$
the following universal
property holds true
\begin{equation}\label{pu}
 \xymatrix{A \times B \ar[drr]^{\varphi} \ar[d] \\ A \tilde{\otimes} B \ar@{.>}[rr]^{\exists!} & & G}
\end{equation}

\begin{remark}  We use the provisory notation $\tilde{\otimes}$ because for abelian groups this definition
differs {\it a priori} from the usual tensor product construction, where one takes the quotient from the free abelian group
generated by $A\times B$. However, we will see that in the abelian case it coincides with the usual tensor product of abelian groups.
\end{remark}

The following proposition is due to Whitney, see \cite[Theorem 11]{Whit}.
It was also reproved by MacHenry in \cite[Theorem 17]{McH} in the same context as ours.
For the sake of completeness, we provide a simple proof of it.
\begin{proposition}[Whitney]\label{EqualitoOfTensorProducts}
$A \tilde{\otimes} B$ is an abelian group and it is isomorphic to $\nicefrac{A}{[A,A]} \otimes \nicefrac{B}{[B,B]}$.
Namely, the tensor product between nonabelian groups is the usual tensor product between the abelianized groups. 
\end{proposition}

\begin{proof} We have the following identities
 $$(a_1 a_2) \tilde{\otimes} (b_2 b_1) = (a_1 a_2 \tilde{\otimes} b_2) (a_1 a_2 \tilde{\otimes} b_1) = (a_1 \tilde{\otimes} b_2) (a_2 \tilde{\otimes} b_2) (a_1 \tilde{\otimes} b_1) (a_2 \tilde{\otimes} b_1);$$
 and
 $$(a_1 a_2) \tilde{\otimes} (b_2 b_1) = (a_1 \tilde{\otimes} b_2 b_1) (a_2 \tilde{\otimes} b_2 b_1) = (a_1 \tilde{\otimes} b_2) (a_1 \tilde{\otimes} b_1) (a_2 \tilde{\otimes} b_2) (a_2 \tilde{\otimes} b_1).$$
Then $(a_1 \tilde{\otimes} b_1) (a_2 \tilde{\otimes} b_2) = (a_2 \tilde{\otimes} b_2) (a_1 \tilde{\otimes} b_1)$.
 Hence $A \tilde{\otimes} B$ is abelian.

On the other hand, the natural arrow $A \times B \to \nicefrac{A}{[A,A]} \otimes \nicefrac{B}{[B,B]}$ is a group homomorphism in each variable
and it is easy to check that it has universal property (\ref{pu}) but only for $G$ abelian.
As the abelian group $A \tilde{\otimes} B$ has the same universal property, they are isomorphic through
the canonical morphism.
\end{proof}

From now on,  $\tilde{\otimes}$ will be denoted by $\otimes$. We will need the next result about commutator subgroups of the free product of two groups. Its proof is elementary and can be found, for instance, in  \cite[Section 1.3, Proposition 4]{serreTRees}.
\begin{lemma}
 \mbox{$[A,B]=\langle \{[a,b]\}_{a\in A, b\in B} \rangle$} is a free subgroup of $A\ast B$ in the generators $[a,b]$, $a\in A$, $b \in B$, $a$, $b \neq 1$. Moreover it is normal in $A\ast B$. 
 \end{lemma}

The following proposition is the main result of the article of MacHenry, \cite{McH}. Arguably, the proof we exhibit here using universal properties is simpler than the one given in \cite{McH}.

\begin{proposition} \label{tensor}
The group $[A,B]^{(2)}$ is isomorphic to  $A \otimes B$.
\end{proposition}

\begin{proof}
Observe first that since $ [A,B]$ is a normal subgroup of $A\ast B$, then $[A\ast B,[A,B]]$ is a subgroup of $[A,B]$. The identity $x[y,z]x^{-1}=[xyx^{-1},xzx^{-1}]$ shows that it is normal.
It follows that $[A,B]^{(2)}=\nicefrac{[A,B]}{[A \ast B, [A,B]]}$.

Let us give explicitly the isomorphisms $\xymatrix{[A,B]^{(2)} \ar@<.5ex>[r]^u & A \otimes B \ar@<.5ex>[l]^v}$.
The application $A \times B \to [A,B]^{(2)}$ given by $(a,b) \mapsto [a,b]$ is an homomorphism in each variable because of Proposition \ref{morfismo}.
This defines $v$.

On the other hand, the map $[A,B] \xto{\tilde u} A \otimes B$ given by $[a,b] \mapsto a \otimes b$
extends to a group homomorphism thanks to the previous lemma. Thus, in order to define $u$, we just need to show that
$\tilde u$ vanishes on $[A \ast B, [A,B]]$.
To this end, take an element  $[p,q]\in [A \ast B, [A,B]]$ with $p \in A \ast B$ and $q \in [A,B]$.
Then $\tilde u ([p,q]) = \tilde u (pqp^{-1}q^{-1}) =
\tilde u (pqp^{-1}) \tilde u (q)^{-1}$. Thus, we must show that $\tilde u (pqp^{-1}) = \tilde u (q)$. If $x \in A$ then
\begin{align*}
\tilde u (x[a,b]x^{-1}) &= \tilde u ([xa,b][b,x]) = (xa \otimes b)(x \otimes b)^{-1} \\
&= (x^{-1} \otimes b)(xa \otimes b) = a \otimes b = \tilde u ([a,b]).
\end{align*}
A similar argument works for $x \in B$. Induction on the length of a word $x\in A\ast B$ shows that $\tilde u (x[a,b]x^{-1}) = \tilde u ([a,b])$.
Since $\{[a,b] :a\in A,\,b\in B \}$ generates $[A,B]$, it follows that $[A \ast B, [A,B]]\subset \ker(\tilde u)$.
In particular $[A,B]^{(2)} \xto{ u} A \otimes B$ is well defined. It is clear now that $uv = id$ and $vu= id$.
\end{proof}
\begin{remark}
Combining it with the exact sequence \eqref{SESeqa} gives the central extension 
\begin{equation} \label{SESeq}
1 \to A\otimes B \to A \stackrel{_2}{\ast} B \to A \oplus B \to 1.
\end{equation}
\end{remark}

\begin{corollary}\label{NilpotentProductEqualDirectSum}
$|A \stackrel{_2}{\ast} B|=|A||B||A\otimes B|$. In particular,
 if a group $A$ is perfect, namely if $A=[A,A]$, 
 then for any group $B$ the second nilpotent product $A \stackrel{_2}{\ast} B$ is isomorphic to the direct product $A \oplus B$.
\end{corollary}
\begin{proof}
 The equality between the cardinals follows from the exact sequence \eqref{SESeq}.
  Moreover, if $A$ is perfect, then Proposition \ref{EqualitoOfTensorProducts} entails that $A\otimes B=\{0\}$. 
  Hence, the exact sequence \eqref{SESeq} gives the required isomorphism.
 \end{proof}

The next proposition identifies the derived group in the second nilpotent product of two groups.
\begin{proposition} \label{[G,G]}
 Let $G=A \stackrel{_2}{\ast} B$. Its commutator subgroup $[G,G]$ is isomorphic to $[A,A] \oplus [B,B] \oplus [A,B]^{(2)}$.
 In particular, $G$ is abelian only when both $A$ and $B$ are abelian and $A\otimes B$ is trivial.
\end{proposition}
\begin{proof}
By Proposition \ref{abc}, a generic commutator element in $G$ is of the form
 $[a_1b_1c_1,a_2b_2c_2]$ where $a_i \in A$, \mbox{$b_i \in B$} and $c_i \in [A,B]^{(2)}$.
This element is equal to $[a_1,a_2][b_1,b_2]c$ for some $c \in [A,B]^{(2)}\subset\mathcal Z(G)$, (in fact \mbox{$c=[a_1,b_2][b_1,a_2]$}). The product of two such elements is of the form 
$$[a_1,a_2][b_1,b_2]c[a_3,a_4][b_3,b_4]d = [a_1,a_2][a_3,a_4][b_1,b_2][b_3,b_4]cd,$$
 (here we used Corollary \ref{commuteswithcommutator}), 
and this is exactly the product  in the direct product $[A,A] \oplus [B,B] \oplus [A,B]^{(2)}$.
The subgroups $[A,A]$, $[B,B]$, $[A,B]^{(2)}$ are contained in $[G,G]$, so we have
an explicit surjective morphism $[A,A] \oplus [B,B] \oplus [A,B]^{(2)} \to [G,G]$.
 It is injective because of Proposition \ref{abc}.
\end{proof}

\vskip .2in
\begin{examples}
\hfill

\begin{enumerate}
\item If $p,q\in \mathbb N$  are coprime numbers, then 
$\nicefrac{\mathbb{Z}}{p\mathbb{Z}}\stackrel{_2}{\ast}\nicefrac{\mathbb{Z}}{q\mathbb{Z}}
\simeq
\nicefrac{\mathbb{Z}}{p\mathbb{Z}} \oplus \nicefrac{\mathbb{Z}}{q\mathbb{Z}}$.
This is  because whenever $(p,q)=1$ the only $\varphi$ that verifies the condition of the diagram \eqref{pu} is $\varphi=0$.
This means that $\nicefrac{\mathbb{Z}}{p\mathbb{Z}} \otimes \nicefrac{\mathbb{Z}}{q\mathbb{Z}}=0$.
Propositions \ref{abc} and \ref{tensor} yields the desired result.

\item \label{Example 2} More generally,  if $p,q\in \mathbb N$  and $d:={\rm gcd}(p,q)$, it follows that
$\nicefrac{\mathbb{Z}}{p\mathbb{Z}} \otimes \nicefrac{\mathbb{Z}}{q\mathbb{Z}}= \nicefrac{\mathbb{Z}}{d\mathbb{Z}}$.
Thus the order of $\nicefrac{\mathbb{Z}}{p\mathbb{Z}} \stackrel{_2}{\ast} \nicefrac{\mathbb{Z}}{q\mathbb{Z}}$ is $pqd$.

\item $\nicefrac{\mathbb{Z}}{n\mathbb{Z}} \stackrel{_2}{\ast} \nicefrac{\mathbb{Z}}{n\mathbb{Z}}$ is isomorphic to the Heisenberg group:
$$\texttt{Heis}\left (\nicefrac{\mathbb{Z}}{n\mathbb{Z}}\right )=\left\{\begin{pmatrix} 1 && a &&c\\0 &&1 && b\\0&& 0&& 1\end{pmatrix}:\,\,a,b,c\in \nicefrac{\mathbb{Z}}{n\mathbb{Z}} \right\}$$
A straightforward way to see this is to verify that the function
$$\Psi: \nicefrac{\mathbb{Z}}{n\mathbb{Z}} \stackrel{_2}{\ast} \nicefrac{\mathbb{Z}}{n\mathbb{Z}}\to \texttt{Heis}\left (\nicefrac{\mathbb{Z}}{n\mathbb{Z}}\right )$$
$$
\Psi(abc)= \begin{pmatrix} 1 && b && c\\0 && 1 && a\\0&& 0&& 1\end{pmatrix}
$$
is a group isomorphism (here we are using Proposition \ref{abc}). Observe that in particular $\nicefrac{\mathbb{Z}}{n\mathbb{Z}} \stackrel{_2}{\ast} \nicefrac{\mathbb{Z}}{n\mathbb{Z}}$ is a non abelian group of order $n^3$ and when $n=2$ this is also isomorphic to the dihedral group $D_4$. The same strategy shows that
$\mathbb{Z} \stackrel{_2}{\ast} \mathbb{Z}$ is isomorphic to the Heisenberg group $\texttt{Heis}\left (\mathbb{Z}\right )$.

\item Denote with $D_n$ the dihedral group of order $2n$, namely the group with presentation $\langle a,b\,| a^n=1, b^2=1, bab=a^{-1}\rangle$.
It is plain that $[D_n,D_n]=\langle a^{-2}\rangle$.  If $n$ is even, $[D_n,D_n]$ has order $n/2$
in which case $\nicefrac{D_n}{[D_n,D_n]}=\langle \bar a,\bar b\rangle\cong \nicefrac{\mathbb{Z}}{2\mathbb{Z}}\oplus\nicefrac{\mathbb{Z}}{2\mathbb{Z}}$. Thus for $n$ even, $D_n \stackrel{_2}{\ast} D_n$ is a (non-abelian) group of order $2^6n^2$ and its derived subgroup is isomorphic to
 $\oplus_1^2(\nicefrac{\mathbb{Z}}{\frac{n}{2}\mathbb{Z}})\,\oplus_1^4 (\nicefrac{\mathbb{Z}}{2\mathbb{Z}})$.
If $n$ is odd, a similar analysis shows that  $D_n \stackrel{_2}{\ast} D_n$ is a (non abelian) group of order $2^3n^2$ and its derived subgroup is isomorphic to $\oplus_1^2\nicefrac{\mathbb{Z}}{n\mathbb{Z}}\oplus \nicefrac{\mathbb{Z}}{2\mathbb{Z}}$.
\end{enumerate}
\end{examples}

\begin{remark} Example (\ref{Example 2}) and Proposition \ref{morfismo} can be used to prove that inside $\nicefrac{\mathbb{Z}}{4\mathbb{Z}}\stackrel{_2}{\ast}\nicefrac{\mathbb{Z}}{2\mathbb{Z}}$, the subgroup generated by the element of order $2$ in $\nicefrac{\mathbb{Z}}{4\mathbb{Z}}$ together with $\nicefrac{\mathbb{Z}}{2\mathbb{Z}}$ is abelian and isomorphic to $\nicefrac{\mathbb{Z}}{2\mathbb{Z}}\oplus \nicefrac{\mathbb{Z}}{2\mathbb{Z}}$. 
This shows that  if $\tilde A\triangleleft A$ and $\tilde B\triangleleft B$, the group $\langle \tilde A, \tilde B\rangle \triangleleft  A\stackrel{_2}{\ast}  B$ is not isomorphic to $\tilde A\stackrel{_2}{\ast}  \tilde B$. 
\end{remark}

\section{Permanence properties of the second nilpotent product of two groups}
The purpose of this section is to prove Theorem (A) from the introduction. We rephrase it here.

\begin{proposition}\label{StabProp}
Let $A$ and $B$ be countable groups. The second nilpotent product of them, $A \stackrel{_2}{\ast} B$, has one of the following properties:
\begin{enumerate}
\item\label{nil} nilpotent;

\item\label{amen} amenable;

\item\label{haag} Haagerup approximation property;

\item \label{exact} exact (or boundary amenable, or satisfies property A of Yu);

\item \label{propT} Kazdhan property (T);

\end{enumerate}
if and only if $A$ and $B$ have the same property.
\end{proposition}
\begin{remark}
\hfill
 \begin{enumerate}
\item[(i)] We refer  to \cite{BDV,BO,CCJJV} for the definitions and thorough treatments of the properties of groups stated above.
\item[(ii)] Golovin showed that the second nilpotent product preserves nilpotency. For the sake of completeness, we include a short proof here.
\item[(iii)] Except for nilpotency, Proposition \ref{StabProp} is obvious and not interesting when the groups at hand are finite. 
\end{enumerate}
\end{remark}
\begin{proof}
Since $A$ and $B$ are subgroups of $G:=A \stackrel{_2}{\ast} B$ and the properties (\ref{nil}), (\ref{amen}), (\ref{haag}) and (\ref{exact}) are inherited by subgroups, it follows that if $G$ satisfies one of these four properties then both $A$ and $B$ must also satisfy it. The fact that Property (T) is inherited by quotients (see  \cite[Theorem 1.3.4 ]{BDV}) tells that if $G$ satisfies (\ref{propT}) then $A\simeq\nicefrac {G}{\overline{B}}$ and $B\simeq\nicefrac {G}{\overline{A}}$ must also satisfy (\ref{propT}), where $\overline{A}$ denotes the normal closure of $A$ in $G$. We are now left to show the reverse implications.

In order to prove (\ref{nil}), let $h\in [G,G]$ and $g\in G$, $g=abc$ as in Proposition \ref{abc}. By Proposition \ref{[G,G]}, $h$ is of the form $h=a_1b_1c_1$, with $a_1\in[A,A],\, b_1\in[B,B],\, c_1\in [A,B]^{(2)}$.
Then $[g,h]=[abc,a_1b_1c_1]=[ab,a_1b_1]=[a,a_1][b,b_1]$. This means that
the third term in the lower central series of $G$,  $G_3:=[G,[G,G]]$ is equal to $[A,[A,A]]\oplus [B,[B,B]]=A_3\oplus B_3$. It follows by induction that for $n>2$ the $n^{th}$ term of the lower central series of $G$ is equal to $A_n\oplus B_n$. In particular, if $A$ and $B$ are nilpotent groups of classes $n$ and $m$, then $G$ is either abelian or nilpotent of class $\max\{2,n,m \}$. 

In order to prove (\ref{amen}), we consider the short exact sequence in (\ref{SESeq}).
If $A$ and $B$ are amenable, then $A\oplus B$ is amenable. Since $A \otimes B$ is abelian, by   \cite[Proposition G.2.2]{BDV} $G$ is amenable.

In order to prove (\ref{haag}), recall that the Haagerup property is preserved by taking subgroups and direct sums, thus by Proposition \ref{[G,G]} the subgroup $[G,G]$ has the Haagerup property if both $A$ and $B$ have it. Since $\nicefrac {G}{[G,G]}$ is abelian, and extensions with amenable quotients preserve the Haagerup property (\cite[Example 6.1.6]{CCJJV}), then $G$ has the Haagerup property.

In order to prove (\ref{exact}),  we consider the short exact sequence in (\ref{SESeq}). Then recall that abelian groups are exact, and that subgroups and extensions of exact groups are exact \cite[Proposition 5.1.11]{BO}. 

Finally, in order to show (\ref{propT}), assume that both $A$ and $B$ have the Property (T). Their abelianizations are finite groups. Then by Proposition \ref{EqualitoOfTensorProducts}, $A\otimes B$ is a finite group. In particular both ends of the short exact sequence (\ref{SESeq}) have the Property (T). We apply then \cite[Proposition 1.7.6]{BDV} to obtain (\ref{propT}).
\end{proof}

\begin{remark} Proposition \ref{abc}, Proposition \ref{tensor} together with the proof of (\ref{nil}) shows that the 2-nilpotent product of finite abelian $p$-groups is a finite $p$-group of nilpotency class $2$. \\
It might seem that (\ref{propT}) could be used to construct property (T) groups with large center. Unfortunately this is not the case since we proved that if $A$ and $B$ have the  property (T) the group $[A,B]^{(2)}$ is finite.
\end{remark}
\section{Second nilpotent product indexed by a set}
In this section we consider an index set $\mathcal I$ and for each $i\in \mathcal I$, a group $H_i$. Recall from the introduction the next definition.
\begin{defi}
For a family of groups $\{H_i\}_{i \in \mathcal I}$ indexed on a set $\mathcal I$, the second nilpotent product of the family is the group
$$\2P {i \in \mathcal I} H_i := \nicefrac{\PR {i \in \mathcal I} H_i}{\big\langle [\PR {i \in \mathcal I} H_i,[H_j,H_k]]_{j \neq k}\big\rangle}$$
\end{defi}
The remainder of this section presents several facts about the 2-nilpotent product of arbitrarily many groups that will be needed to prove Corollary B and Theorems C and D. This will also enable us to give a short proof of the associativity of the second nilpotent product, the only one among the five properties listed in the introduction and proved by Golovin that does not follow immediately from the definitions.

\begin{proposition} The second nilpotent product is functorial.  Explicitly, a family of morphisms $H_i \to K_i$ induce a natural morphism
\mbox{$\2P {i \in \mathcal I} H_i \to \2P {i \in \mathcal I} K_i$}.
\end{proposition}

\begin{proof}
This is a straightforward consequence of the functoriality of the free product.
\end{proof}

A universal property for the second nilpotent product is given by the the following diagram
\begin{equation}
 \xymatrix{H_i \ar[drr]^{r_i} \ar[d] \\ \2P{i\in \mathcal I}{H_i} \ar@{.>}[rr]^{\exists!} & & G}
\end{equation}
\noindent where $G$ is a group and the morphisms $r_i$ verify $[r_i(g),r_j(h)] \in \mathcal Z(G)$ if $i \neq j$.

\vskip .2in
The group generated by the commutators $[h_i,h_j]$ with $h_i\in H_i$, $h_j\in H_j$, $i\neq j$
is central in $\2P {i \in \mathcal I} H_i$. Let fix a total order on $\mathcal I$. We have that
$$\langle[h_i,h_j] | h_i\in H_i, h_j\in H_j, i \neq j \rangle= \langle[H_i,H_j]^{(2)} | i \neq j \rangle =
\bigoplus_{\substack{i, j \in \mathcal I \\ i< j}} [H_i,H_j]^{(2)}.$$
This, together with Proposition \ref{tensor} immediately imply the next generalization of Proposition \ref{tensor}.
\begin{proposition}\label{sumoftensors} 
Let $\{H_i\}_{i \in \mathcal I}$ be a family of groups indexed on an ordered set $\mathcal I$.
The subgroup $\langle[h_i,h_j] | h_i\in H_i, h_j\in H_j, i \neq j \rangle$ of the group $\2P {i \in \mathcal I} H_i $ is central and it is 
isomorphic to  $$\bigoplus_{\substack{i, j \in \mathcal I \\ i< j}} H_i \otimes H_j.$$
\end{proposition}

Observe that because of functoriality, for any subset $\mathcal S \subset \mathcal I$ the natural projection
$\2P {i \in \mathcal I} H_i \xto {\pi_{\mathcal S}} \2P {i \in \mathcal S} H_i$
is a well defined group homomorphism. Moreover the projection $\pi_\mathcal S(x)$ can be computed in any word $x$ by erasing all letters belonging to a group $H_i$ whose index $i \notin \mathcal S$.
Thus, for $\mathcal S \subset \mathcal J \subset \mathcal I$ the composition of the two natural projections coincides with projecting
directly to $\mathcal S$.

For two elements $i \neq j \in \mathcal I$ we denote with $\pi_{(i,j)}$ the projection
\mbox{$\2P {l \in \mathcal I} H_l \xto{\pi_{(i,j)}} H_i \stackrel{_2}{\ast} H_j$}.
\vskip .1in

\begin{proposition} \label{escritura}
 Fixing a total order in $\mathcal I$, every element $x \in \2P {i \in \mathcal I} H_i$ admits a unique
 representation $$x = a_{i_1}a_{i_2}...a_{i_l}\,\omega$$

\noindent where $a_{i_k}\in H_{i_k}$, $i_1<i_2<...<i_l\in\mathcal I$ , and
$\omega \in\mathcal Z(\2P {i \in \mathcal I} H_i)$
is of the form $$\omega=\prod_{\substack{i_{k}< i_{r}\\ r\leq l}} c_{i_{k},i_{r}}
\text { with }c_{i_{k},i_{r}}\in[H_{i_{k}},H_{i_{r}}]^{(2)} .$$
\end{proposition}

\begin{proof}

\noindent  Existence: Start from a word $P$ that represents $x$. By rearranging the elements (adding the
  corresponding commutators) we can obtain such a representation.
Uniqueness: by Proposition \ref{abc}, the projection $\pi_{(i_{k},i_{r})}(x)$ determines $a_{i_k}$, $a_{i_r}$ and $c_{i_{k},i_{r}}$.
\end{proof}

This result, combined with Proposition \ref{sumoftensors}, can be used to compute the order of the second nilpotent product of finitely many finite groups.
\begin{example} $\2P {1\leq i\leq n} \nicefrac{\mathbb{Z}}{p\mathbb{Z}}$ is the universal nil-2 exponent $p$ group in $n$ generators.
It has order $p^{\frac {n^2+n}{2}}$, and its derived subgroup has order $p^{\frac {n^2-n}{2}}$ and it is isomorphic to 
$\bigoplus_{1}^{\frac{n^2-n}{2}}\nicefrac{\mathbb{Z}}{p\mathbb{Z}}$.
\end{example}

\begin{proposition}  \label{morfismo - caso general}
Let $\mathcal J\subset \mathcal I$. For $x \in \2P {j \in \mathcal J} H_j$, the commutator $[x,-]$ defines a group homomorphism
$$\2P {i \in \mathcal I \setminus \mathcal J} H_i  \xto {[x,-]} \bigoplus_{{i \in \mathcal I\setminus \mathcal J,\, j\in\mathcal J}} [H_i,H_j]^{(2)}\ \subset \mathcal Z(\2P {i\in \mathcal I} H_i).$$
\end{proposition}

\begin{proof}
Since the identity $[x,yz] = [x,y]y[x,z]y^{-1}$ is always valid, it is enough to show that $[x,z]\in \mathcal Z(\2P {i \in \mathcal I} H_i)$ whenever $x \in \2P {j \in \mathcal J} H_j$ and $z\in\2P {i \in \mathcal I \setminus \mathcal J} H_i$.
Assume first that $x\in H_j$, $j\in\mathcal J$. Any $z\in \2P {i \in \mathcal I \setminus \mathcal J} H_i$ can be represented as a finite word $z_{i_1}z_{i_2}...z_{i_l}$, where for all $1\leq k\leq l$, $z_{i_k}\in H_{i_k}$ and  $i_k\not \in \mathcal J$.
 Since for all $k$, $[x,z_{i_k}]\in \mathcal Z(\2P {i \in \mathcal I} H_i)$, induction on $l$ shows that $[x,z]=[x,z_{i_1}z_{i_2}...z_{i_l}]=[x, z_{i_1}z_{i_2}...z_{i_{l-1}}][x,z_{i_l}]\in\mathcal
Z(\2P {i \in \mathcal I} H_i)$. Repeating the same induction argument but now for $x \in \2P {j \in \mathcal J} H_j$ finishes the proof.
\end{proof}

We can now give a short proof of the associativity of the second nilpotent product.

\begin{proposition}\cite[Golovin]{Gol1} \label{assoc}
  Let $S = \underset{i\in \mathcal I}{\coprod}S_i$ be a disjoint union of index sets, and $H_j$ a group for each $j \in S$.
  Then $$\2P {j \in S} H_j \simeq \2P {i \in \mathcal I} \Big(\2P {j \in S_i} H_j \Big)$$
\end{proposition}

\begin{proof}
The isomorphism will be given by the identity on each $H_j$. As these generate both groups, the only nontrivial fact is that they are well defined.
We will induce them with the help of the universal property of the second nilpotent product. Natural inclusions give the following diagram

 $$\xymatrix{H_k \ar[d]_{r_k} \ar[rr]^{l_k} & & \2P{j \in {S_{i(k)}}} H_j \ar[dll]_{s_{i(k)}} \ar[d]^{t_{i(k)}} \\
 \2P {j \in S} H_j \ar@<.7ex>@{.>}[rr]^{u} &  &  \2P {i \in \mathcal I} \Big(\2P {j \in S_i} H_j \Big) \ar@<.7ex>@{.>}[ll]^{v} } $$

In order to see that $u$ is well defined, we must check that for every $\alpha \in H_{k_1}$ and every $\beta \in H_{k_2}$ with $k_1 \neq k_2$, the element
  $[t_{i(k_1)} l_{k_1}(\alpha), t_{i(k_2)} l_{k_2}(\beta)]$ belongs to the center of 
$\2P {i \in \mathcal I} \Big(\2P {j \in S_i} H_j \Big)$. In the case when $i(k_1) \neq i(k_2)$, 
this follows from the definition of the second nilpotent product of the family of groups $\left\{\2P {j \in S_i}H_j\right\}_{i\in\mathcal I}$.
 We are left to analyze the case when $i(k_1) = i(k_2)$.
In this case, 
 $[t_{i(k_1)} l_{k_1}(\alpha), t_{i(k_1)} l_{k_2}(\beta)]= t_{i(k_1)}\left([ l_{k_1}(\alpha), l_{k_2}(\beta)]
\right)$.
The definition of the second nilpotent product of the family $\left\{H_j\right\}_{ j\in i(k_1)}$ entails that 
our element is in the center of $\2P{j \in {S_{i(k_{1})}}} H_j$.
 On the other hand, if in Proposition \ref{morfismo - caso general} we take $\mathcal J=\mathcal I\setminus\left\{i(k_{1})\right\}$,
 it follows that for any $x\in  \2P{i \in \mathcal J} \Big(\2P {j \in S_i} H_j \Big)$
  we have that
 \begin{align*}
& [x,[t_{i(k_1)} l_{k_1}(\alpha), t_{i(k_1)} l_{k_2}(\beta)]]=\\
&= [x,t_{i(k_1)} l_{k_1}(\alpha)][x, t_{i(k_1)} l_{k_2}(\beta)]
 [x,(t_{i(k_1)} l_{k_1}(\alpha))^{-1}][x, (t_{i(k_1)} l_{k_2}(\beta))^{-1}]\\
 &= [x,t_{i(k_1)} l_{k_1}(\alpha)] [x,(t_{i(k_1)} l_{k_1}(\alpha))^{-1}][x, t_{i(k_1)} l_{k_2}(\beta)]
[x, (t_{i(k_1)} l_{k_2}(\beta))^{-1}]\\
&=1.
 \end{align*}
 All this combined implies that $[t_{i(k_1)} l_{k_1}(\alpha), t_{i(k_1)} l_{k_2}(\beta)]$ is central in the group
 $\2P {i \in \mathcal I} \Big(\2P {j \in S_i} H_j \Big)$.

In order to see that $v$ is well defined, we must  check that $[s_{i_1}(\gamma), s_{i_2}(\delta)] \in \mathcal Z \Big( \2P {j \in S} H_j \Big)$, where
 \mbox{$\gamma \in \2P {j \in S_{i_1}} H_j$}, $\delta \in \2P {j \in S_{i_2}} H_j$, $i_1 \neq i_2$. 
To that end, consider $\gamma=\gamma_{j_1}\ldots\gamma_{j_n}$ with $j_1,\cdots,j_n\in S_{i_{1}}$ and $\delta =  \delta_{j^\prime_{1}}\ldots, \delta_{j^\prime_{m}}$
with $j^{\prime}_{1},\cdots,j^\prime_{m}\in S_{i_{2}}$.  Once again, by Proposition \ref{morfismo - caso general}, we have
$$[s_{i_1}(\gamma), s_{i_2}(\delta)] = \prod_{\substack{1\leq k\leq n\\1\leq l\leq m}} [\gamma_{j_k},\delta_{j^{\prime}_{l}}],$$ 
and this product clearly belongs to the center.
\end{proof}
We are now in position to prove Corollary \ref{amenability of many products} from the introduction.
\begin{proof}[Proof of Corollary B] If $\mathcal I$ is finite, the result follows from associativity together with Proposition \ref{StabProp}. If $\mathcal I=\mathbb N$, then 
$$\2P {i\in\mathbb N } H_i=\bigcup_{n\in\mathbb N} \2P {i\in\{1,2,\ldots,n\} } H_i$$
and amenability, the Haagerup property and exactness are preserved under countable increasing unions of discrete groups (see \cite[Proposition G.2.2]{BDV}, \cite[Proposition 6.1.1]{CCJJV} and \cite[Exercise 5.1.1]{BO}). 
\end{proof}

\begin{remark} Property (T) is not stable under taking the second nilpotent product of infinitely many discrete groups.
 This is because such a group is not finitely generated.
\end{remark}

\section{second nilpotent wreath products}\label{sec5}

Let $H$ and $G$ be two countable groups. We consider the second nilpotent product of 
$|G|$-many copies of $H$ indexed by $G$, namely, we consider the group
$\2P {g \in G} H_g$, where for each $g$, $H_g=H$. 

Since the shift action of $G$ on the free product $\PR {G} H_g$ leaves
the set $\big \{ [\PR {g \in  G} H_g,[H_h,H_k]]_{h \neq k}\big \}$ invariant, it follows that this action passes to the factor group $\nicefrac{\PR {g \in G} H_g}{\big\langle [\PR {g \in  G} H_g,[H_h,H_k]]_{h \neq k}\big\rangle}$. In other words $G$ acts on $\2P {g\in G} H_g$. 
\begin{defi}The semi-direct product  
$$\Big (\2P {G}H\Big )\rtimes G$$  will be called {\it the restricted second nilpotent wreath product} of $H$ and $G$.
\end{defi}

A variant of this construction that motivated this article appeared in \cite[Section 5]{SasTor1}. 
The goal of this section is to show Theorem (C) 
and Theorem (D) 
 from the introduction.

\begin{defi}  \label{defsop}
The support of an element $x \in \2P {i \in \mathcal I} H_i$ is the subset of $\mathcal I$ whose elements are all the indices  $i$ such that one of the elements $a_i$, $c_{ij}, j\neq i$  in the representation of $x$ as in Proposition \ref{escritura}, is nontrivial. Equivalently, for every $x\neq e$,
$$supp(x)=\left \{i\in\mathcal I: \text{ \rm there exists } j\neq i \text{ \rm such that } \pi_{(i,j)}(x)\notin H_j \right\}.$$
The support of the identity element is the empty set.
\end{defi}
\begin{remark}\label{gauge} Some obvious properties of the support are:
\begin{enumerate}
\item $supp(x)$ is a finite set;
\item $supp(x)=supp(x^{-1})$;
\item $supp (xy)\subset supp (x)\cup supp(y)$;
\item \label{cuatro}when the index set is a group $G$,  
 it follows that $$supp(g.x)=g.supp(x),\text{ for all }g\in G
\text { and for all }x\in \2P {G} H.$$
\end{enumerate}
\end{remark}
The proof of Theorem (C) that we will exhibit here follows the general strategy developed by Cornulier et.al. in \cite{CSV2}.
In fact, the result will follow after setting up the premisses that allow us to apply  \cite[Theorem 5.1]{CSV2}. For the sake of completeness, we will include its statement. Let us first recall the next definition.

\begin{defi} \cite[Definition 3.3]{CSV2} Let $W$ be a group, and $X$ be a set. $\mathcal A=2^{(X)}$ denotes the set of finite subsets of $X$. A $W$-invariant $\mathcal A$-gauge on $W$ is a function $\psi: W\to\mathcal A$ such that 

$$\psi(w)=\psi(w^{-1})\,\,\,\,\,\,\,\forall w\in W;$$
 $$\psi (ww')\subset \psi (w)\cup \psi(w')\,\,\,\,\,\,\,\forall w,w\in W.$$
\end{defi}

\begin{example}  The first three items of Remark \ref{gauge} say that 
the support function is a $\2P {i \in \mathcal I} H_i$-invariant $ 2^{(\mathcal I)}$-gauge on $\2P {i \in \mathcal I} H_i$. While condition  (\ref{cuatro}) of Remark \ref{gauge} says that when the index set is a group $G$, the support function is $G$-equivariant.
\begin{theorem}\cite[Theorem 5.1]{CSV2}. Let $W, G$ be groups, with $G$ acting on $W$ by automorphisms. Set $\mathcal{A}=2^{(G)}$. Let $\psi$ be a left $W$-invariant, G-equivariant $\mathcal A$-gauge on $W$. Assume that there exists a $G$-invariant conditionally negative definite function $u$ on $W$ such that, for every finite subset $F\subset G$, the restriction of $u$ to every subset of the form $W_F:=\{w\in W:\psi(w)\subset F\}$ is proper. Then $W\rtimes G$ is a Haagerup group if and only if $G$ is Haagerup.
\end{theorem}
\end{example}
So, in order to prove Theorem C from the introduction, it is enough to show the next proposition.
\begin{proposition}
Let $G$ and $H$ be discrete countable groups. Assume $H$ Haagerup. Then there exists a $G$-invariant and conditionally
negative definite function $\2P {g \in G} H \xto{u} \R$ such that for every finite subset $F\subset G$ the restriction of $u$ to any subset of the form 
$\{x\in \2P {g \in G} H: supp(x)\subset F\}$
is  proper.
\end{proposition}

\begin{proof}
Since $H$ is Haagerup, by item (\ref{haag}) of Proposition \ref{StabProp}, $H \stackrel{_2}{\ast} H$ is Haagerup. By definition, this means that there exists a  proper conditionally negative definite function $\varphi:H \stackrel{_2}{\ast} H \to \R_{\geq 0}$. 
For $h, k \in G, h \neq k$,  we have
$$\2P {G} H \xto{\pi_{(h,k)}} H_h \stackrel{_2}{\ast} H_k \xto{\cong} H \stackrel{_2}{\ast} H \xto{\varphi} \R.$$
Denote $v_{(h,k)}:= \varphi\circ \pi_{(h,k)}$.
Let \mbox{$\Psi:G \to \mathbb{N}$} be an enumeration of $G$.\\
{\bf Claim:}
The function $$u= \sum_{\underset{h\neq k}{h,k \in G}}{\frac{1}{2^{\Psi(h^{-1}k)}}\, v_{(h,k)}}$$ satisfies the required conditions.\\

\begin{enumerate}
\item[(i)] $u$ is $G$-invariant: Since $\pi_{(h,k)}(g.x)=\pi_{(g^{-1}h,g^{-1}k)}(x)$ it follows that
$$u(g.x)=\sum_{\underset{h\neq k}{h,k \in G}}{\frac{1}{2^{\Psi(h^{-1}k)}}\, v_{(h,k)}(g.x)}=\sum_{\underset{h\neq k}{h,k \in G}}{\frac{1}{2^{\Psi((g^{-1}h)^{-1}(g^{-1}k))}}\, v_{(g^{-1}h,g^{-1}k)}(x)}=u(x).$$

\item[(ii)]For every fixed $x$, $u(x)$ is finite: First notice that for all $h,k\notin supp(x)$, $v_{(h,k)}(x)=\varphi(e)=0$.
Then
\begin{align}\label{finitesum}
u(x)=&\sum_{\underset{h\neq k}{h,k \in supp(x)}}{\frac{1}{2^{\Psi(h^{-1}k)}}\, v_{(h,k)}}(x)+
\sum_{\substack{h \in supp(x)\\ k\notin supp(x)}}{\frac{1}{2^{\Psi(h^{-1}k)}}\, v_{(h,k)}}(x)+\\
 &\sum_{\substack{h \notin supp(x)\\ k\in supp(x)}}{\frac{1}{2^{\Psi(h^{-1}k)}}\, v_{(h,k)}}(x). \nonumber
\end{align}
Since $supp(x)$ is finite, the first summand in (\ref{finitesum}) is finite. \\
Take $h\in supp(x)$  fixed. Then
for all $k, k' \notin supp(x)$, $\pi_{(h,k)}(x)\in H_h$ and  $\pi_{(h,k^{\prime})}(x)\in H_h$. 
Hence, 
$\pi_{(h,k)}(x)= \pi_{(h,k^{\prime})}(x)$, 
and then $v_{(h,k)}(x)=v_{(h,k')}(x)$.
It follows that the sum
 $$\sum_{k\notin supp(x)}{\frac{1}{2^{\Psi(h^{-1}k)}}\, v_{(h,k)}}(x)$$
 is convergent  for all $h\in supp(x)$. Hence, the second summand in (\ref{finitesum}) is finite.
The same method shows that the third summand in (\ref{finitesum}) is finite.

\item [(iii)] Restrictions are proper:
Fix a finite subset $F \subset G$.
 Let
$N=\underset{h,k \in F}{\max} \hspace{2pt} \Psi(h^{-1}k)$.
Let $a, b \in F$, $a \neq b$. Then for any $x\in \HP H$  we have the inequalities
$$ v_{(a,b)}(x) \leq \sum_{\underset{h\neq k}{h,k \in F}} v_{(h,k)}(x) \leq 2^N u(x).$$
This means that the set 
$$\{x \in \HP H : supp(x) \subset F, u(x) \leq M \}$$
is contained in
$$\{x \in \HP H : supp(x) \subset F,  v_{(a,b)}(x) \leq 2^N M \text{ for all } a,b\in F \}.$$
This set is finite since $\varphi$ is proper and for all $x\neq x'$ whose supports are contained in $F$ there exists $a,b\in F$ such that $\pi_{(a,b)}(x)\neq\pi_{(a,b)}(x')$.

\item[(iv)] $u$ is a conditionally negative definite function (c.n.d.f.): This is obvious since the set of c.n.d.f. is a convex cone, and pointwise limit of c.n.d.f. is a c.n.d.f. (see, for instance, \cite[Proposition C.2.4]{BDV}.)
\end{enumerate}
\end{proof}

\begin{remark}
The case when $H$ is finite and $G=\mathbb F_2$ could be shown by mimicking the proof given in \cite{CSV1} for the lamplighter group.
This alternative approach has the advantage of being self contained since it only requires to
transfer the {\it space with walls} structure from $\mathbb F_2$ to $(\2P {\mathbb F_2} H) \rtimes \mathbb F_2$.
\end{remark}

\begin{proof}[Proof of Theorem D]
If $A$ is abelian and $G$ is amenable, Corollary \ref{amenability of many products} implies that
the restricted second nilpotent wreath product of $A$ and $G$ is amenable and thus unitarizable.
To show the converse let $\{g_i: g_i\in G\}_{i\in\mathbb N}$ be an enumeration of $G$.
Since
$C=\bigoplus_{ i< j} [A_{g_i},A_{g{_j}}]^{(2)}$
is a $G$-invariant central subgroup of $\2P {G} A$, then $\tilde C=(C,e_G)$ is a normal subgroup of $\Big(\2P {G} A \Big) \rtimes G$. Their quotient is:
$$\Big(\2P {G} A\Big) \rtimes G\Big/\tilde C\simeq\Big(\2P{G} A\,\big/ C\Big) \rtimes G\simeq\Big(\bigoplus_{G} A\Big)\rtimes G.$$
    When $G$ is non-amenable,   \cite[Theorem 1]{MoOz} says that $\Big(\bigoplus_{G} A\Big)\rtimes G$ is non unitarizable. Since a quotient of a unitarizable group must be unitarizable, it follows that $\Big(\2P {G} A\Big) \rtimes G$ is non unitarizable.
\end{proof}

\textbf{Acknowledgments.}This work was partially supported by the grant PIP-CONICET 11220130100073CO. We thank the anonymous referee for her or his 
detailed reading of the manuscript and for the many suggestions that certainly helped to improve the exposition of this article. 
\bibliographystyle{amsplain}

\bibliography{dosnilpotente}

\end{document}